\documentclass{amsart}

\usepackage[all,cmtip]{xy}
\usepackage{amsmath}
\usepackage{mathabx}
\usepackage{color}
\definecolor{blue3}{RGB}{0,0,205}
\usepackage[colorlinks,citecolor=blue3,linkcolor=blue3,urlcolor=blue3,pdfpagemode=UseNone,backref = page]{hyperref}

\newtheorem{theorem}{Theorem}[section]

\newtheorem{corollary}[theorem]{Corollary}

\theoremstyle{definition}
\newtheorem{definition}[theorem]{Definition}
\newtheorem{example}[theorem]{Example}

\theoremstyle{remark}
\newtheorem{remark}[theorem]{Remark}

\numberwithin{equation}{section}

\newcommand{\RR}{\mathbb{R}}
\newcommand{\ZZ}{\mathbb{Z}}

\newcommand{\bb}{\bar{b}}

\newcommand{\cF}{\mathcal{F}}

\newcommand{\fX}{\mathfrak{X}}

\begin{document}

\title{A splitting theorem for compact Vaisman manifolds}


\author{Giovanni Bazzoni}
\address{Mathematisches Institut der Ludwig-Maximilians-Universit\"at, M\"unchen\\ Theresienstr.\ 39, 80333, M\"unchen}
\curraddr{}
\email{bazzoni@math.lmu.de}
\thanks{}

\author{Juan Carlos Marrero}
\address{Departmento de Matem\'aticas, Estad{\'\i}stica e Investigaci\'on Operativa, Secci\'on de Matem\'aticas, Facultad de Ciencias, Universidad de La Laguna, 38271, La Laguna, Tenerife, Spain}
\curraddr{}
\email{jcmarrer@ull.edu.es}
\thanks{}

\author{John Oprea}
\address{Department of Mathematics\\
Cleveland State University\\
Cleveland OH \\
44115  USA}
\email{j.oprea@csuohio.edu}
\thanks{}



\date{}

\dedicatory{In memory of Sergio Console}

\begin{abstract}
We extend to metric compact mapping tori a splitting result for coK\"ahler manifolds. In particular, we prove that a compact Vaisman manifold is finitely covered by the product of a Sasakian manifold and a circle.
\end{abstract}

\maketitle

\section{Introduction}\label{Sec:1}

It is quite often the case that two geometric structures are intimately related to one another.
This is true, for instance, for Sasakian and K\"ahler structures. Indeed, assume that $K$ is a K\"ahler manifold such that the K\"ahler class $[\omega]$ is integral. 
The Boothby-Wang construction (\cite{Boothby_Wang}) produces a principal bundle $S^1\to S\to K$ with a connection whose curvature is $\omega$; moreover, the total space $S$ admits a Sasakian structure. 
This construction can be reversed if $S$ is compact and the Sasakian structure is regular. Moreover, given a manifold $S$ endowed with an almost contact metric structure, the product $S\times\RR^{>0}$ with the cone metric 
is K\"ahler if and only if the structure is Sasakian.

But more is true. Sasakian structures are also related to Vaisman structures: if $\varphi$ is an automorphism of a Sasakian manifold $S$, then the mapping torus $S_\varphi$ has a natural Vaisman structure.
Conversely, Ornea and Verbitsky showed in \cite{OV3} that a compact Vaisman manifold is always diffeomorphic to the mapping torus of an automorphism of a Sasakian manifold (see Example \ref{ex:2}).

In this short note we propose to explore further the relation between Vaisman and Sasakian structures. We recalled that a compact Vaisman manifold $V$ is a mapping torus $S_\varphi$ of a Sasakian automorphism. The idea is to 
apply the techniques of \cite{BO} to show that the structure group of such a mapping torus is finite, hence a compact Vaisman manifold is finitely covered by the 
product of a compact Sasakian manifold and a circle. From this, we obtain topological information about compact Vaisman manifolds.

\section{Preliminaries}

\begin{definition}
 Let $X$ be a topological space and let $\varphi\colon X\to X$ be a homeomorphism. The \emph{mapping torus} or \emph{suspension} of $(X,\varphi)$, denoted $X_\varphi$, is the quotient space
 \[
  \frac{X\times [0,1]}{(x,0)\sim(\varphi(x),1)}.
 \]
The pair $(X,\varphi)$ is the \emph{fundamental data} of $X_\varphi$.
\end{definition}

Notice that $\mathrm{pr}_2\colon X\times[0,1]\to[0,1]$ induces a projection $\pi\colon X_\varphi\to S^1$, where $S^1=\RR/\ZZ$, whose fiber is $X$. Hence $X_\varphi$ is a fiber bundle with base $S^1$ and fiber $X$.
It can be shown (compare \cite[Proposition 6.4]{BO}) that the structure group of the bundle $X\to X_\varphi\to S^1$ is the cyclic group 
$\langle\varphi\rangle\subset \mathrm{Homeo}(X)$. Moreover, $X_\varphi$ is trivial as a bundle over $S^1$, i.e.\ $X_\varphi \cong X\times S^1$ over $S^1$, if and only if $\varphi$ lies in the connected component of the identity 
of $\mathrm{Homeo}(X)$.

Here is an equivalent definition: given the fundamental data $(X,\varphi)$, we consider the following $\ZZ$-action on the product $X\times\RR$:
\[
 m\cdot(x,t)=(\varphi^m(x),t+m).
\]
Notice that this action is free and properly discontinuous. The quotient space $(X\times\RR)/\ZZ$ is homeomorphic to $X_\varphi$. In particular, if $M$ is a smooth manifold and $\varphi\colon M\to M$ is a diffeomorphism, we 
conclude that $M_\varphi$ is a smooth manifold and $\pi\colon M_\varphi\to S^1$ is a smooth fiber bundle with fiber $M$. The length 1-form $\sigma$ on $S^1$ pulls back under $\pi$ to a closed 1-form $\theta\in\Omega^1(M_\varphi)$. 
Actually, since $[\sigma]\in H^1(S^1;\ZZ)$, $[\theta]\in H^1(M_\varphi;\ZZ)$ and $[\theta]$ itself gives the map $\pi$ under the usual correspondence $H^1(M_\varphi,\ZZ)\equiv [M_\varphi,S^1]$.

\begin{definition}
 Let $(M,g)$ be a Riemannian manifold and let $\varphi\colon M\to M$ be an isometry. We call $M_\varphi$ the \emph{metric mapping torus} of $(M,g,\varphi)$.
\end{definition}

A metric mapping torus $M_\varphi$ has a natural Riemannian metric, best described if one thinks of $M_\varphi$ as a quotient of $M\times\RR$. Indeed, consider the product metric $\tilde{h}=g+d t^2$ on $M\times\RR$. 
Then, since $\ZZ$ acts by isometries on $(M\times\RR,\tilde{h})$, the metric $\tilde{h}$ descends to a metric $h$ on the mapping torus $M_\varphi$. The vector field $\partial_t$ on $M\times\RR$ maps to the tangent vector field to 
$S^1$ under the map $M\times\RR\to M_\varphi\stackrel{\pi}{\to} S^1$. Moreover, the 1-form $\theta$ is unitary and parallel. Thus, if we consider the standard Riemannian metric on $S^1$, $\pi\colon M_\varphi\to S^1$ is a Riemannian 
submersion with totally geodesic fibers.

\begin{definition}
 We call $h$ the \emph{adapted metric} on the metric mapping torus $M_\varphi$.
\end{definition}

Let $(g,J,\omega)$ be a Hermitian structure on a manifold $V$ of dimension $2n+2$, $n\geq 1$. Associated to it is the \emph{Lee form}, defined by
\[
 \theta=-\frac{1}{n}\delta\omega\circ J;
\]
here $\delta$ is the co-differential. The Hermitian structure $(g,J,\omega)$ is \emph{K\"ahler} if $\omega$ is parallel. In particular, $\theta=0$ in this case. The structure $(g,J,\omega)$ is \emph{Vaisman} if the Lee form is 
non-zero and parallel and $d\omega=\theta\wedge\omega$. In fact, we will assume, without the loss of generality, that $\theta$ is unitary. Thus, a Vaisman structure is a particular case of a \emph{locally conformal K\"ahler} structure, 
where $\theta$ is only required to be closed with $d\omega=\theta\wedge\omega$. Note that, if $n\geq 2$, the last condition implies the closedness of the Lee form. Locally conformal K\"ahler geometry is a very active area of research 
(see \cite{DO,OV,OV2}) and has recently attracted interest in Physics (see \cite{Shahbazi1,Shahbazi2}). 
Locally conformal K\"ahler manifolds with parallel Lee $1$-form were studied for the first time by Vaisman in \cite{Vaisman1}. In a subsequent paper (see \cite{Vaisman}), Vaisman discussed these structures under the name of 
\emph{generalized Hopf structures}. Indeed, the main example of a compact Vaisman manifold is the Hopf manifold $S^{2n+1}\times S^1$.

Let $S$ be an odd-dimensional manifold. Consider an almost contact metric structure $(\xi,\eta,g,\phi)$ on $S$ and let $\Omega$ be the K\"ahler form\footnote{Notice that we denote $\omega$ the K\"ahler form of a Hermitian
structure and $\Omega$ the K\"ahler form of an almost contact metric structure.} (see \cite{Blair} for an exposition on almost contact metric geometry).
The structure is \emph{coK\"ahler} if $\nabla\Omega=0$; it is possible to show that $\nabla\eta=0$ in this case, hence, in particular, $d\eta=0$.
\begin{example}\label{ex:1}
 If $(g,J,\omega)$ is a K\"ahler structure on a manifold $K$ and $\varphi\colon K\to K$ is a holomorphic isometry, then $K_\varphi$ is a so-called K\"ahler mapping torus, providing an example of a coK\"ahler manifold. 
 The metric on $K_\varphi$ is the adapted one. Note that, in this case, the parallel 1-form giving the map $K_\varphi\to S^1$ is $[\eta]\in H^1(K_\varphi;\ZZ)$.
 Conversely, a compact coK\"ahler manifold is diffeomorphic to a K\"ahler mapping torus (see \cite{Li}).
\end{example}

An almost contact metric structure $(\xi,\eta,g,\phi)$ is \emph{Sasakian} if $d\eta=\Omega$ and $N_\phi+2d\eta\otimes\xi=0$, where $N_\phi$ is the Nijenhuis torsion of $\phi$. Let $S$ be a manifold endowed with a Sasakian structure. 
A diffeomorphism $\varphi\colon S\to S$ is a \emph{Sasakian automorphism} if $\varphi^*\eta=\eta$ and $\varphi^*g=g$.

A Sasakian manifold $S$ is endowed with a 1-dimensional foliation $\cF_\xi$, the \emph{characteristic foliation}, whose tangent sheaf is generated by $\xi$. The foliation $\cF_\xi$ is Riemannian and transversally K\"ahler, 
see \cite[Section 7.2]{BoGa}. A $p$-form $\alpha\in\Omega^p(S)$ is \emph{basic} if $\imath_\xi\alpha=0$ and $\imath_\xi d\alpha=0$. We denote basic forms by $\Omega^*(\cF_\xi)$; this is a differential subalgebra of $\Omega^*(S)$. 
Its cohomology, called the \emph{basic cohomology} of $\cF_\xi$, is denoted $H^*(\cF_\xi)$. We collect the most relevant features of the basic cohomology:

\begin{theorem}[{\cite[Proposition 7.2.3 and Theorem 7.2.9]{BoGa}}]\label{basic_cohomology}
 Let $S$ be a compact manifold of dimension $2n+1$ endowed with a Sasakian structure $(\xi,\eta,g,\phi)$. Then:
 \begin{itemize}
  \item the groups $H^p(\cF_\xi)$ are finite dimensional, $H^{2n}(\cF_\xi)\cong\RR$ and $H^p(\cF_\xi)=0$ for $p>2n$;
  \item $[d\eta]^p\in H^{2p}(\cF_\xi)$ is non-trivial for $p=1,\ldots,n$;
  \item the map $L^p\colon H^{n-p}(\cF_\xi)\to H^{n+p}(\cF_\xi)$, $[\alpha]\mapsto [(d\eta)^p\wedge\alpha]$ is an isomorphism for $0\leq p\leq n$.
 \end{itemize}
\end{theorem}

Recall that a connected commutative differential graded algebra $(A,d)$ is \emph{cohomologically K\"ahlerian} if its cohomological dimension is even (say $2n$), it satisfies Poincar\'e duality and there exists a 2-cocycle $\omega$ 
such that the map $H^{n-p}(A)\to H^{n+p}(A)$, $[\alpha]\mapsto[\omega^{n-p}\wedge\alpha]$, is an isomorphism for $0\leq p\leq n$. As a consequence of Theorem \ref{basic_cohomology}, $H^*(\cF_\xi)$, considered as a 
commutative differential graded algebra with trivial differential, is cohomologically K\"ahlerian.

\begin{example}\label{ex:2}
 Let $(\xi,\eta,g,\phi)$ be a Sasakian structure on a manifold $S$ and let $\varphi\colon S\to S$ be a Sasakian automorphism. The mapping torus $S_\varphi$ has a Vaisman structure with the adapted metric.
 In \cite[Structure Theorem]{OV}, Ornea and Verbitsky claimed that every compact manifold endowed with a Vaisman structure is the mapping torus of a Sasakian manifold and a Sasakian automorphism. 
 In \cite{OV3} they argued that this is actually imprecise, but provided a modified version of this statement. In \cite[Corollary 3.5]{OV3}, they proved that if a compact manifold $V$ admits a Vaisman structure, then $V$ admits 
 another Vaisman structure which arises as mapping torus of a Sasakian manifold and a Sasakian automorphism. 
 Thus, up to diffeomorphism, every compact Vaisman manifold is the mapping torus of a Sasakian manifold and a Sasakian automorphism.
 \end{example}

\section{Main result}\label{sec:3}

Consider a metric mapping torus with adapted metric, $(M_\varphi,h)$, and let $\theta\in\Omega^1(M_\varphi)$ be the closed 1-form described in Section \ref{Sec:1}. In Examples \ref{ex:1} and \ref{ex:2}, $\theta$ is not only a 
closed form, but it is also unitary and \emph{parallel} with respect to the Levi-Civita connection of the adapted metric. More generally, we can suppose we are given a mapping torus $M_\varphi$ with a Riemannian metric $\bar{h}$
such that $\theta\in\Omega^1(M_\varphi)$ is a parallel 1-form. In such a case, $M_\varphi$ is locally isometric to the product $M\times\RR$ and it follows that $\theta$ is unitary and parallel.
There is therefore no loss of generality in assuming that $\bar{h}$ is the adapted metric.

We prove the following result:

\begin{theorem}\label{theo:1}
 Let $(M,g)$ be a compact Riemannian manifold, let $\varphi\colon M\to M$ be an isometry and let $(M_\varphi,h)$ be the mapping torus with the adapted metric. Let $\theta\in\Omega^1(M_\varphi)$ be the 
 unitary and parallel 1-form. Then there is a finite cover $p\colon M\times S^1\to M_\varphi$, the deck 
 group is isomorphic to a finite group $\ZZ_m$, for some $m>0$, and acts diagonally and by translations on the second factor. We have a diagram of fiber bundles
 \[
  \xymatrix{
  M\ar[r]\ar[d]_{\cong} & M\times S^1\ar[r]\ar[d]^p & S^1\ar[d]^{\cdot m}\\
  M\ar[r] & M_\varphi\ar[r] & S^1
  }
 \]
 and $M_\varphi$ fibers over the circle $S^1/\ZZ_m$ with finite structure group $\ZZ_m$.
\end{theorem}

\begin{proof}
 The proof is basically a reproduction of the argument used in \cite{BO} for coK\"ahler manifolds. We recall it briefly. The first step is to notice that $\upsilon\in \fX(M_\varphi)$, the metric dual of the 1-form $\theta$, 
 is a unitary and parallel vector field and, in particular, Killing. If the metric $h$ happens to be adapted, then $\upsilon$ is the image under the derivative of the projection $M\times\RR\to M_\varphi$ of the vector field $\partial_t$.
 By the Myers-Steenrod theorem (see \cite{Myers-Steenrod}), $\mathrm{Isom}(M_\varphi,h)$ is compact, 
 so the closure of the flow of $\upsilon$ in $\mathrm{Isom}(M_\varphi,h)$ is a torus $T$. This gives a free $T$-action on $M_\varphi$. Choose a vector field $\hat{\upsilon}$ in the Lie algebra of $T$, close enough to $\upsilon$, 
 and such that $\hat{\upsilon}$ generates a circle action on $M_\varphi$. At some point $x_0\in M_\varphi$ we surely have $\theta(\hat{\upsilon})(x_0)\neq 0$, since $\theta(\upsilon)(x_0)\neq 0$. But being $\theta$ harmonic and 
 $\hat{\upsilon}$ Killing, this implies that $\theta(\hat{\upsilon})\neq 0$. We assume henceforth that $\theta(\hat{\upsilon})>0$ and denote by $\tau\subset M_\varphi$ an orbit of this $S^1$-action. Then
 \begin{equation}\label{eq:1}
  \int_\tau\theta=\int_0^1\theta\left(\frac{d\tau}{dt}\right)dt=\int_0^1\theta(\hat{\upsilon})dt>0.
 \end{equation}
 Consider the orbit map $\alpha\colon S^1\to M_\varphi$, $g\mapsto g\cdot x_0$ and the composition
 \[
  H_1(S^1;\ZZ)\xrightarrow{\alpha_\ast} H_1(M_\varphi;\ZZ)\xrightarrow{\pi_\ast} H_1(S^1;\ZZ).
 \]
 We remarked above that, under the correspondence $H^1(M_\varphi;\ZZ)\equiv [M_\varphi,S^1]$, $\pi$ is given by $\theta$; thus \eqref{eq:1} tells us that $\pi_*$ is non-zero when evaluated on an element of $H_1(M_\varphi;\ZZ)$ coming 
 from the orbit map. Since $H_1(S^1;\ZZ)=\ZZ$, this means that $\alpha_*$ is injective. We conclude that the $S^1$-action is homologically injective.
 
 The second step consists in relating the homological injectivity of this $S^1$-action with the reduction of the structure group of the bundle $M_\varphi\to S^1$ to a finite group and the existence of a finite cover of $M_\varphi$ with 
 the desired properties. This uses the more general notion of transversal equivariance of a fibration over a torus with respect to a smooth torus action, developed by Sadowski in \cite{Sadowski}. 
 These ideas were developed first in the topological context by Conner and Raymond, see \cite{Conner-Raymond}. We refer to \cite{BO} for a detailed explanation of the result.
 \end{proof}
 
 \begin{remark}\label{Kotschick}
  Prof.\ Dieter Kotschick has suggested to us that Theorem \ref{theo:1} can be proven in an easier way, without appealing to the results of Conner - Raymond and Sadowski. 
  Indeed, by the Myers-Steenrod theorem, the isometry group $\mathrm{Isom}(M,g)$ of a 
  compact Riemannian manifold $(M,g)$ is a compact Lie group; in particular, it has a finite number of connected components. This implies that if $\varphi$ is an isometry of $(M,g)$, there exists an integer $m>0$ such that 
  $\varphi^m$ belongs to the connected component of the identity $\mathrm{Isom}_0(M,g)$ of $\mathrm{Isom}(M, g)$. Indeed, if $\varphi^n \not\in \mathrm{Isom}_0(M, g)$, for every integer $n>0$, 
  then $[\varphi^n] \neq [\varphi^m]$ in the quotient group $\mathrm{Isom}(M,g)/\mathrm{Isom}_0(M, g)$, for $n \neq m$. However, this is not possible, since $\mathrm{Isom}(M,g)/\mathrm{Isom}_0(M, g)$ is just the finite group of 
  connected components of $\mathrm{Isom}(M,g)$. Consider now the map $\gamma_m\colon S^1\to S^1$ given by $\gamma_m(z)=z^m$ and use it to pull back to the first $S^1$ the fiber bundle 
  $M\to M_\varphi\to S^1$. It is clear that the structure group of $\gamma_m^*M_\varphi$ is generated by $\varphi^m$, hence $\gamma_m^*M_\varphi\cong M\times S^1$. This gives the splitting up to final cover. One also obtains 
  an action of the finite group $\mathbb{Z}_m$, generated by the isotopy class of $\varphi$, on the product $M\times S^1$, which is diagonal and by translations on $S^1$. 
  In \cite{BO} we overlooked this simple approach and appealed rather to the techniques of Conner - Raymond and Sadowski since our main goal was to investigate the rational-homotopic properties of compact coK\"ahler manifolds. Among the 
  outcomes of this research, we quote the proof of the toral rank conjecture for compact coK\"ahler manifolds (see \cite{BLO}).
 \end{remark}

 \begin{corollary}\label{cor:1}
  Let $(M,g)$ be a compact Riemannian manifold, let $\varphi\colon M\to M$ be an isometry and let $(M_\varphi,h)$ be the mapping torus with the adapted metric. Then
  \[
 H^*(M_\varphi;\RR)\cong H^*(M;\RR)^G\otimes H^*(S^1;\RR),
  \]
  where $G\cong \ZZ_m$ and $m$ is the smallest positive integer such that $\varphi^m \in \mathrm{Isom}_0(M, g)$.
 \end{corollary}
 
 \begin{proof}
  For a finite $G$-cover $\tilde{X}\to X$, one has $H^*(X;\RR)\cong H^*(\tilde{X};\RR)^G$.
 \end{proof}

 \begin{corollary}\label{cor:2}
  Let $V$ be a compact Vaisman manifold. Then there exists a finite cover $p\colon S\times S^1\to V$, where $S$ is a compact Sasakian manifold, the deck group is isomorphic to $\ZZ_m$, for some $m>0$, acts diagonally and by 
  translations on the second factor. We have a diagram of fiber bundles
  \[
  \xymatrix{
  S\ar[r]\ar[d]_{\cong} & S\times S^1\ar[r]\ar[d]^p & S^1\ar[d]^{\cdot m}\\
  S\ar[r] & V\ar[r] & S^1
  }
 \]
 and $V$ fibers over the circle $S^1/\ZZ_m$ with finite structure group $\ZZ_m$.
 \end{corollary}
 
 \begin{proof}
  By the aforementioned result of Ornea and Verbitsky (see \cite[Corollary 3.5]{OV3}), $V$ is diffeomorphic to a mapping torus $S_\varphi$ where $S$ is a compact Sasakian manifold and $\varphi\colon S\to S$ is a Sasakian automorphism. 
  Under this identification, the Lee form $\theta$, which is parallel by definition on a Vaisman manifold, gives the projection $V\cong S_\varphi\to S^1$. It is now enough to apply Theorem \ref{theo:1}.
 \end{proof}
 
 In the next corollary, we show how to apply our splitting theorem to obtain well-known results on the topology of compact Vaisman manifolds. Compare \cite{Vaisman}.
 
 \begin{corollary}\label{cor:3}
  Let $V$ be a compact connected Vaisman manifold of dimension $2n+2$ and let $b_r(V)$ be the $r^{th}$ Betti number of $V$. Then $b_p(V)-b_{p-1}(V)$ is even for $p$ odd and $1\leq p\leq n$. In particular, $b_1(V)$ is odd.
 \end{corollary}
 \begin{proof}
  By \cite[Corollary 3.5]{OV3}, $V$ is diffeomorphic to a mapping torus $S_\varphi$, where $\varphi$ is an automorphism of the compact Sasakian manifold $S$. 
  Let $(\xi,\eta,g,\phi)$ be the Sasakian structure of $S$ and let $\cF_\xi$ denote the characteristic foliation; then $(H^*(\cF_\xi),0)$ is a cohomologically K\"ahlerian algebra with 
  K\"ahler class $[d\eta]\in H^2(\cF_\xi)$. If $G\cong\ZZ_m$ is the finite structure group of the finite cover $p\colon S\times S^1\to S_{\varphi}$, then the $G$-action preserves $\cF_\xi$, since $\varphi$, which generates the structure group of the mapping torus, is a Sasakian automorphism.
  Hence we can consider the invariant basic cohomology $H^*(\cF_\xi)^G$.
  For the basic K\"ahler class, it holds $[d\eta]\in H^2(\cF_\xi)^G$. 
  There is an exact sequence (see \cite[page 215]{BoGa})
  \begin{equation}\label{ses:1}
   \cdots\to H^p(S;\RR)\to H^{p-1}(\cF_\xi)\xrightarrow{e} H^{p+1}(\cF_\xi)\to H^{p+1}(S;\RR)\to\cdots
  \end{equation}
  where $e$ is the multiplication by the basic K\"ahler class $[d\eta]$. Since $\dim V=2n+2$, the cohomological dimension of $H^*(\cF_\xi)$ is $2n$. This implies that the map 
  $e\colon H^{p-1}(\cF_\xi)\to H^{p+1}(\cF_\xi)$ is injective for $0\leq p\leq n$, hence \eqref{ses:1} splits and gives short exact sequences
  \begin{equation}\label{ses:2}
   0\to H^{p-1}(\cF_\xi)\xrightarrow{e} H^{p+1}(\cF_\xi)\to H^{p+1}(S;\RR)\to 0, \quad 0 \leq p\leq n.
  \end{equation}
  Notice that \eqref{ses:2} is a short exact sequence of $G$-modules.
  Since we are working with real coefficients, every term in \eqref{ses:2} is a real vector space.
  By taking invariants, we obtain short exact sequences
  \begin{equation}\label{ses:3}
   0\to H^{p-1}(\cF_\xi)^G\to H^{p+1}(\cF_\xi)^G\to H^{p+1}(S;\RR)^G\to 0, \quad 0 \leq p\leq n;
  \end{equation}
  the surjectivity of $H^{p+1}(\cF_\xi)^G\to H^{p+1}(S;\RR)^G$ follows by averaging over $G$. By \cite[Proposition 2.3]{BLO}, $H^*(\cF_\xi)^G$ is also 
  cohomologically K\"ahlerian. We now set $\bb_p(S)\coloneq \dim H^p(S;\RR)^G$ and $\bb_p(\cF_\xi)\coloneq \dim H^p(\cF_\xi)^G$; then $\bb_p(\cF_\xi)$ is even for $p$ odd. In view of Corollary \ref{cor:1}, we have, for $1\leq p\leq n$,
  \[
   b_p(V)=\bb_p(S)+\bb_{p-1}(S)=\bb_p(\cF_\xi)-\bb_{p-2}(\cF_\xi)+\bb_{p-1}(\cF_\xi)-\bb_{p-3}(\cF_\xi),
  \]
  hence $b_p(V)-b_{p-1}(V)=\bb_p(\cF_\xi)-2\bb_{p-2}(\cF_\xi)+\bb_{p-4}(\cF_\xi)$. If $p$ is odd, then $b_p(V)-b_{p-1}(V)$ is even.
 \end{proof}

 According to \cite{OV2}, the fundamental group of a compact Vaisman manifold $V$ sits in an exact sequence
 \[
  0\to G\to\pi_1(V)\to\pi_1(X)\to 0
 \]
where $\pi_1(X)$ is the fundamental group of a K\"ahler orbifold and $G$ is a quotient of $\ZZ^2$ by a subgroup of rank $\leq 1$. We give a different characterization of the fundamental group of a Vaisman manifold: 

\begin{corollary}
  Let $V$ be a compact Vaisman manifold. Then $\pi_1(V)$ has a subgroup of finite index of the form $\Gamma\times\ZZ$, where $\Gamma$ is the fundamental group of a compact Sasaki manifold.
 \end{corollary}
 \begin{proof}
 It is enough to consider the finite cover $S\times S^1\to S_\varphi$.
 \end{proof}
 
Sasaki groups, and their relation with K\"ahler groups, have been investigated for instance in \cite{Chen,Kasuya}.

Another application of the splitting theorem is to the group of automorphisms of a compact Sasakian manifolds. Let $S$ be a compact manifold endowed with a Sasakian structure $(\xi,\eta,g,\phi)$, let $\mathrm{Aut}(\xi,\eta,g,\phi)$ be 
the group of Sasakian automorphisms and let $\varphi\in\mathrm{Aut}(\xi,\eta,g,\phi)$. Form the mapping torus $S_\varphi$. By Corollary \ref{cor:2}, we find $m>0$ such that 
$\varphi^m\in\mathrm{Aut}_0(\xi,\eta,g,\phi)$, the identity component of $\mathrm{Aut}(\xi,\eta,g,\phi)$.

\begin{corollary}
 If $S$ is a compact manifold endowed with a Sasakian structure $(\xi,\eta,g,\phi)$, then every element of the group $\mathrm{Aut}(\xi,\eta,g,\phi)/\mathrm{Aut}_0(\xi,\eta,g,\phi)$ has finite order.
\end{corollary}

For a careful analysis of the automorphism group of a Sasakian manifold we refer to \cite{BoGa}.

\begin{remark}
In \cite{BG}, the notion of a $K$-cosymplectic structure was introduced. This is an almost contact metric structure 
 $(\xi,\eta,g,\phi)$ with $d\eta=0$, $d\Omega=0$ and $L_\xi g=0$. One can prove that, in this case, the 1-form $\eta$ is parallel. Examples of K-cosymplectic manifolds are given by mapping tori of almost K\"ahler manifolds 
 $(K,g,J,\omega)$ with a diffeomorphism $\varphi\colon K\to K$ such that $\varphi^*g=g$ and $\varphi^*\omega=\omega$; hence they are metric mapping tori. 
If $K$ is compact, so is $K_\varphi$. Hence we can apply Theorem \ref{theo:1} and obtain a finite cover $K\times S^1\to K_\varphi$.
%
\end{remark}

\section*{Acknowledgements}
We thank Prof.\ Dieter Kotschick for very useful conversations and for pointing out to us that an easier proof of our main result is available.
The second author was partially supported by MICINN (Spain) grant MTM2012-34478. The third author was partially supported by a grant from the Simons Foundation: (\#244393 to John Oprea).

\end{document}